\documentclass{amsart}
\newtheorem{theorem}{Theorem}[section]
\newtheorem{lemma}[theorem]{Lemma}
\newtheorem{proposition}[theorem]{Proposition}
\newtheorem{corollary}[theorem]{Corollary}
\theoremstyle{definition}

\theoremstyle{remark}

\numberwithin{equation}{section}

\begin{document}

\title[1-dimensional centro-affine curvature flow]{Centro--affine curvature flows on centrally symmetric convex curves}
\author[M. N. Ivaki]{Mohammad N. Ivaki}
\address{Department of Mathematics and Statistics,
  Concordia University, Montreal, QC, Canada, H3G 1M8}
\curraddr{}
\email{mivaki@mathstat.concordia.ca}
\thanks{}

\date{}

\dedicatory{}

\begin{abstract}
We consider two types of $p$-centro affine flows on smooth, centrally symmetric, closed convex planar curves, $p$-contracting, respectively, $p$-expanding. Here $p$ is an arbitrary real number greater than $1$. We show that, under any $p$-contracting flow, the evolving curves shrink to a point in finite time and the only homothetic solutions of the flow are ellipses centered at the origin. Furthermore, the normalized curves with enclosed area $\pi$ converge, in the Hausdorff metric, to the unit circle modulo $SL(2)$. As a $p$-expanding flow is, in a certain way, dual to a contracting one, we prove that, under any $p$-expanding flow, curves expand to infinity in finite time, while the only homothetic solutions of the flow are ellipses centered at the origin. If the curves are normalized as to enclose constant area $\pi$, they display the same asymptotic behavior as the first type flow and converge, in the Hausdorff metric, and up to $SL(2)$ transformations, to the unit circle. At the end, we present a new proof of $p$-affine isoperimetric inequality, $p\geq 1$, for smooth, centrally symmetric convex bodies in $\mathbb{R}^2$.
\end{abstract}

\maketitle
\section{Introduction}\label{intro}
The affine normal flow is a widely recognized
evolution equation for hypersurfaces in which each point moves with
velocity given by the affine normal vector. This evolution equation
is the simplest affine invariant flow in differential geometry and
it arises naturally if one considers families of $\delta$-convex
floating bodies of a convex body \cite{BA4}, \cite{S1}. On the
applicability aspect, the affine normal flow appears in image
processing as a fundamental smoothing tool \cite{ST2, ST3, ST1}. It also
provides a nice proof of the Santal\'{o} inequality for smooth convex
hypersurfaces and, respectively, for the
 classical affine isoperimetric inequality, both due to Andrews \cite{BA2}.
 The affine normal evolution has also been implicitly deployed by Stancu in \cite{S2,S1}
  for a breakthrough towards the homothety conjecture for convex floating bodies by Sch\"{u}tt-Werner \cite{SW}.
  In this paper, we consider an extension of the affine normal flow, namely the $p$-centro affine flow
  introduced by Stancu \cite{S}, and we investigate its behavior. The $p$-centro affine flows are natural
  generalizations of the affine normal flow in a way which will be explained below.

Let $K$ be a compact, centrally symmetric, strictly convex body, smoothly embedded in $\mathbb{R}^2$.
We denote the space of such convex bodies by $\mathcal{K}_{sym}$. Let
 $$x_K:\mathbb{S}^1\to\mathbb{R}^2,$$ be
 the Gauss parametrization of $\partial K$, the boundary of $K\in
 \mathcal{K}_{sym}$, where the origin of the plane is chosen to coincide with the center of
 symmetry of the body.
The support function of $\partial K$ is defined by
 $$s_{\partial K}(z):= \langle x_K(z), z \rangle,$$
 for each $z\in\mathbb{S}^1$.
 We denote the curvature of $\partial K$ by $\kappa$ and, furthermore, the radius of curvature of the curve
 $\partial K$ by $\mathfrak{r}$, both as functions on the unit circle. The latter is related to the support function by
 $$\mathfrak{r}[s](z):=\frac{\partial^2}{\partial \theta^2}s(z)+s(z),$$
 where $\theta$ is the angle parameter on $\mathbb{S}^1.$
 Let $K_0\in \mathcal{K}_{sym}$. We consider a family $\{K_t\}_t\in \mathcal{K}_{sym}$, and their  associated smooth
 embeddings $x:\mathbb{S}^1\times[0,T)\to \mathbb{R}^2$, which are evolving according  to the $p$-centro affine flow, namely,
 \begin{equation}\label{e: faree}
 \frac{\partial}{\partial t}x:=-\left(\frac{\kappa}{s^3}\right)^{\frac{p}{p+2}-\frac{1}{3}}\kappa^{\frac{1}{3}}\, z,~~
 x(\cdot,0)=x_{K_0}(\cdot),~~ x(\cdot ,t)=x_{K_t}(\cdot)
 \end{equation}
 for a fixed $p\geq 1$.

 The flow (\ref{e: faree}) without the assumption of symmetry was defined in \cite{S}, in all dimensions in the class of
 $C^2_{+}$ convex bodies having origin in their interiors, for the purpose of finding new global centro-affine invariants of
 smooth convex bodies. Stancu obtains many interesting isoperimetric type inequalities via short time existence of the flow.
 Furthermore, her $p$-flow approach led to a geometric interpretation of the $L_{\phi}$ surface area recently introduced by
 Ludwig and Reitzner \cite{Ludwig}. The case $p=1$, the well-known affine normal flow,  was already addressed by Andrews \cite{BA2} in
 arbitrary dimension, by Sapiro and Tannebaum \cite{ST4} for convex planar curves, and by Angenent, Sapiro and Tannebaum
 \cite{AST} for non-convex curves.  In \cite{ST4}, it was proved that the flow evolves any initial strictly convex curve,
 not necessarily symmetric, until it shrinks to an {\em elliptical} point. Andrews, in \cite{BA2},
 investigated completely this case for hypersurfaces of any dimension and showed that the normalized flow evolves
 any initial strictly convex hypersurface exponentially fast, in the $C^{\infty}$ topology, to an ellipsoid.
 He also proves, in \cite{BA4}, that any convex initial bounded open set shrinks to a point in finite time under the  affine
 normal flow. In \cite{AST}, the authors prove convergence to a point under the affine normal flow starting from any $C^{2}$
 planar curve, not necessarily convex, despite the fact that affine differential geometry in not defined
 for non-convex curves or hypersurfaces. In another  direction, interesting results for the affine normal
flow have been obtained in \cite{LT} by Loftin and Tsui regarding ancient solutions, and existence and regularity of solutions on non-compact strictly convex hypersurfaces. It is necessary pointing out that the case $p=1$, in contrast to the case $p>1$, is the only instance
when the flow (\ref{e: faree}) is no longer {\em anisotropic}. Moreover, the main difference between $p=1$  and the other cases
is that, for $p>1$ the flow is sensitive to the origin (as the name $p$-centro affine suggests)
 meaning that Euclidean translations of an initial curve will lead to different solutions, since translations affect the support function of the convex body which appears in (\ref{e: faree}). The translation invariancy of a flow is a main ingredient to prove
 the convergence to a point \cite{BA5, BA2, BA1, BA4}.  We overcome these issues, and other difficulties in the study of the asymptotic behavior
 of this flow, by restricting it to $\mathcal{K}_{sym}$ and implementing $p$-affine isoperimetric inequalities developed by
 Lutwak \cite{Lutwak}. This approach emphasizes the usefulness of the $p$-affine surface area and
 $p$-affine isoperimetric inequalities which have also been successfully employed by Lutwak and Oliker in \cite{LO} for
 obtaining regularity of the solutions to a generalization of Minkowski problem.
 See \cite{CL, LYZ1, LYZ2, LYZ3, LYZ4, LYZ5, TW1,TW2} for more applications of these invaluable tools.

 We will give  now a description of (\ref{e: faree}) in terms of $SL(2)$-invariant quantities.
  Although it relies on introducing new notations, it provides an alternate view in which (\ref{e: faree}) is naturally an
  affine invariant flow. Equation
   (\ref{e: faree}) is precisely
   $$\frac{\partial}{\partial t}x:=\kappa_0^{\frac{p}{p+2}}\, \mathfrak{n}_0,$$
 where $n_0$ denotes the centro-affine normal vector field along $\partial K$, and, in every direction, $\displaystyle \kappa_0=\frac{\kappa}{s^3}$ is the centro-affine curvature along the boundary. The  centro-affine normal vector field is, pointwise, a multiple of the affine normal vector field $\mathfrak{n}$ which is known to be transversal to the boundary of $K$, but not necessarily orthogonal to it. More precisely, $\mathfrak{n}_0=\kappa_0^{-\frac{1}{3}}\, \mathfrak{n}$. In this paper, we choose to work with the flow's definition as a time-dependent anisotropic flow by powers of the Euclidean curvature and we will resort to the affine differential setting only for a technical step in the study of the normalized evolution.

 Finally, note that the solution of (\ref{e: faree}) remains in $\mathcal{K}_{sym}$, as $s$ and $\kappa$ are {\em symmetric} in the sense
  $$\forall \theta:\ s(\theta+\pi)=s(\theta),~~\kappa(\theta+\pi)=\kappa(\theta).$$
 Here and thereafter, we identify $z=(\cos \theta, \sin \theta)$ with the normal angle $\theta$ itself. We will give a proof of the fact that $K_t\in\mathcal{K}_{sym}$ as long as the flow exists in Lemma \ref{lem: symmetry preserved}.

 We can rewrite the evolution equation (\ref{e: faree}) as a scalar parabolic equation for the support functions on the unit circle:
 $$ s:\mathbb{S}^1\times[0,T)\to \mathbb{R}^{+}$$
 \begin{equation}\label{e: asli}
\frac{\partial}{\partial
t}s=-s^{1-\frac{3p}{p+2}}\mathfrak{r}^{-\frac{p}{p+2}}, \
s(\cdot,0)=s_{\partial K} (\cdot),\  s(\cdot,t)=s_{\partial K_t}(\cdot),
\end{equation}
leading, in general, to an anisotropic planar evolution. As in \cite{BA1}, it can be shown that there is a one-to-one
correspondence between the solutions of (\ref{e: faree}) and those of (\ref{e:
asli}).

 The convex body $K^{\circ}$, dual to the convex body $K\in \mathcal{K}_{sym}$, is defined by
$$K^{\circ}=\{y\in\mathbb{R}^2; |\langle y,x\rangle|\leq 1~~\mbox{for every } x\in K\}.$$
In \cite{S}, the following expanding $p$-centro affine
flow  was defined in connection to (\ref{e: faree})
\begin{equation}\label{e: dual flow}
 \frac{\partial}{\partial t}x:=s\left(\frac{\kappa}{s^{3}}\right)^{-\frac{p}{p+2}}\, z,~~
 x(\cdot,0)=x_{K^{\circ}_0}(\cdot),~~ x(\cdot ,t)=x_{K^{\circ}_t}(\cdot).
\end{equation}
It is easy to check as
$K_t$, evolves according to (\ref{e: faree}), then $K^{\circ}_t$ evolves according to (\ref{e: dual flow}).
Equivalently, the support function of $\partial K^{\circ}_t$, $s_{\partial K^{\circ}_t}$, evolves according to
\begin{equation}\label{e:  dual asli}
\frac{\partial}{\partial
t}s=s^{1+\frac{3p}{p+2}}\mathfrak{r}^{\frac{p}{p+2}},\ \ p \geq 1
\end{equation}
with initial condition $s(\cdot,0)=s_{K^{\circ}_0}(\cdot),$
see Lemma \ref{lem: dual flow}.

At a point $p$ of $\partial K$, the centro-affine curvature mentioned earlier is inversely proportional to the square of the area of
the  centered osculating ellipse at $p$. The centro-affine curvature is thus constant along ellipses centered at the origin which are, therefore, evolving homothetically by (\ref{e: asli}), respectively (\ref{e:  dual asli}).
Coupled with the fact that these flows increase the product $Area (K) \cdot Area (K^{\circ})$ which
is known to reach the maximum  for ellipses centered at the origin (Santal\'o inequality) and the applications of $p$-flow stated above it was natural to investigate the asymptotic behavior of the flows which a priori suggests convergence to ellipses.
While this was the first objective of the paper, in the process we obtained  sharp affine isoperimetric type
inequalities. The latter is related to the $p$-affine surface area introduced by Lutwak in \cite{Lutwak} which has been the subject of intense research since then, see \cite{Ludwig} for a recent, outreaching work which motivates even the present work. Finally, to the best of our knowledge, this is the first study of an anisotropic curvature flow with time-dependent weight. We regard as weight, as well as anisotropic factor, a power of the support function of the evolving body.
In this paper we prove the following two theorems:
\begin{theorem}
Let $p>1$. Let $x_{K_0}:\mathbb{S}^1\to\mathbb{R}^2$ be a smooth, strictly convex embedding of $K_0\in\mathcal{K}_{sym}.$ Then there exists a unique solution $x:\mathbb{S}^1\times [0,T)\to\mathbb{R}^2$ of equation (\ref{e: faree}) with initial data $x_{K_0}$. The solution remains smooth and strictly convex on $[0,T)$ for a finite time $T$ and it converges to the origin of the plane. The rescaled curves given by the embeddings $\sqrt{\frac{\pi}{A_t}}x(\theta,t)$ converge in the Hausdorff metric to the unit circle modulo $SL(2).$
\end{theorem}

 \begin{theorem}
Let $p>1$. Let $x_{K_0}:\mathbb{S}^1\to\mathbb{R}^2$ be a smooth, strictly convex embedding of $K_0\in\mathcal{K}_{sym}.$ Then there exists a unique solution $x:\mathbb{S}^1\times [0,T)\to\mathbb{R}^2$ of equation (\ref{e: dual flow}) with initial data $x_{K_0}.$ The solution remains smooth and strictly convex on $[0,T)$ for a finite time $T$ and it expands in all directions to infinity. The rescaled curves given by the embeddings $\sqrt{\frac{\pi}{A_t}}x(\theta,t)$ converge in the Hausdorff metric to the unit circle modulo $SL(2).$
\end{theorem}
The paper is structured as follows. The next section focuses on the $p$-contracting affine flow. We show that the evolving curves shrink to a point in finite time. To study the convergence of solutions, we resort to the affine differential geometry in the third section. In this section we will obtain a sharp affine isoperimetric inequality along the flow. In the fourth section, we obtain a crucial result about the constant asymptotic value of the centro-affine curvature of any solution. It is here where we conclude the convergence of solutions to a circle modulo $SL(2)$. In the fifth section, we present the relation between the contracting and the expanding flows. Consequently, we deduce an analogous asymptotic behavior for the $p$-expanding affine flow. Finally, in last section, we present a new proof of $p$-affine isoperimetric inequality, $p\geq 1$, for smooth, centrally symmetric convex bodies in $\mathbb{R}^2$.

\section{Convergence to a point and homothetic solutions}

This section is devoted to the contracting $p$-centro-affine
curvature flow. In what follows, by evolving curves we mean the
curves that enclose the evolving convex bodies in
$\mathcal{K}_{sym}.$ We start by recalling two results from \cite{S} whose
proofs are obtained by standard methods employed for geometric PDEs.
\begin{proposition}[Short-term Existence and Uniqueness]\label{prop: short time existence and uniqueness}
Let $K_0$ be a convex body belonging to $\mathcal{K}_{sym}$ and let
$p\ge 1$. Then, there exists a time $T>0$ for which equation (\ref{e: faree})
 has a unique  solution starting from $K_0.$
\end{proposition}
\begin{proposition}[Containment Principle] \label{prop: containment principle}
 If $K_{in}$ and $K_{out}$ are the two convex
bodies in $\mathcal{K}_{sym}$ such that $K_{int}\subset K_{out}$,
and $p\ge 1$, then $K_{in}(t)\subseteq K_{out}(t)$ for as long as
the solutions $K_{in}(t)$ and $K_{out}(t)$ of (\ref{e: faree}) $($with given initial data
$K_{in}(0)=K_{in}$, $K_{out}(0)=K_{out}$$)$   exist
in $\mathcal{K}_{sym}.$
\end{proposition}
\begin{lemma}\label{lem: symmetry preserved}Let $\{K_t\}_t$ be a solution of (\ref{e: faree}) where $K_0\in\mathcal{K}_{sym}.$
Then, $K_t\in\mathcal{K}_{sym}$ as long as the flow exists.
\end{lemma}
\begin{proof}
 Note that both $-x(\cdot+\pi,t)$ and $x(\cdot,t)$ satisfy (\ref{e: faree}) with initial data $-x(\cdot+\pi,0)$ and $x(\cdot,0)$, respectively. At time $t=0$ we have $-x(\cdot+\pi,0)=x(\cdot,0)$. Therefore, by Proposition \ref{prop: short time existence and uniqueness}
we conclude that $-x(\cdot+\pi,t)=x(\cdot,t)$ as long as the flow exists.

\end{proof}
The following evolution equations can be derived by a direct
computation.
\begin{lemma}\label{lem: evolution equations} Under the flow (\ref{e:
asli}), one has
\begin{equation}\label{e: evolution equation}\frac{\partial}{\partial t}\mathfrak{r}=-\frac{\partial^2}{\partial\theta^2}\left(s^{1-\frac{3p}{p+2}}\mathfrak{r}^{-\frac{p}{p+2}}\right)-
s^{1-\frac{3p}{p+2}}\mathfrak{r}^{-\frac{p}{p+2}},
\end{equation}
\begin{equation}\label{e: volume}
\frac{d}{dt}A(t)=-\Omega_p(t),
\end{equation}
where
$A(t):=A(K_t)=\displaystyle \frac{1}{2} \int_{\mathbb{S}^1}\frac{s}{\kappa}\, d\theta$
is the area enclosed by the evolving curve, hence the area of $K_t$,
and
$\Omega_p(t):=\Omega_p(K_t)=\displaystyle \int_{\mathbb{S}^1}\frac{s}{\kappa}\left(\frac{\kappa}{s^3}\right)^{\frac{p}{p+2}}d\theta,$
is the p-affine length of $\partial K_t$.
\end{lemma}

In trying to prove the convergence of the evolving curves to a point, the main obstacle was that, except for the case $p=1$,  we could not find a uniform lower bound on the curvature of evolving curves. However, we could show, with several fruitful consequences, that there exists an entire family of increasing quantities related to the speed of the flow,
$s^{1-\frac{3p}{p+2}}\mathfrak{r}^{-\frac{p}{p+2}}$.
\begin{proposition} \label{cor: strengthned speed} For $1\leq q\leq \frac{2p}{p+1}$, or
$q=0$, the flow (\ref{e: asli}) increases in time\\
$$\displaystyle\min_{\theta\in\mathbb{S}^1}\left(s^{q}\left(\frac{\kappa}{s^3}\right)^{\frac{p}{p+2}}\right)(\theta,t).$$
\end{proposition}
\begin{proof}
Using the evolution equations  (\ref{e: asli}) and
(\ref{e: evolution equation}), we obtain
\begin{align}\label{e: laplacian}
\frac{\partial}{\partial t}\left(s^{q-\frac{3p}{p+2}}\mathfrak{r}^{-\frac{p}{p+2}}\right)&=\left(\frac{\partial}{\partial t}s^{q-\frac{3p}{p+2}}\right)\mathfrak{r}^{-\frac{p}{p+2}}+s^{q-\frac{3p}{p+2}}\frac{\partial}{\partial t}\mathfrak{r}^{-\frac{p}{p+2}}\nonumber\\
&=-\left(q-\frac{3p}{p+2}\right)s^{q-\frac{3p}{p+2}-1}\mathfrak{r}^{-\frac{p}{p+2}}\left(s^{1-\frac{3p}{p+2}}
\mathfrak{r}^{-\frac{p}{p+2}}\right)\\
&~~~~+\frac{p}{p+2}\mathfrak{r}^{-\frac{p}{p+2}-1}s^{q-\frac{3p}{p+2}}\left[
\left(s^{1-\frac{3p}{p+2}}\mathfrak{r}^{-\frac{p}{p+2}}\right)_{\theta\theta}+s^{1-\frac{3p}{p+2}}\mathfrak{r}^{-\frac{p}{p+2}}\right] \nonumber\\
&=\left(\frac{3p}{p+2}-q\right)s^{q-\frac{6p}{p+2}}\mathfrak{r}^{-\frac{2p}{p+2}}+\frac{p}{p+2}s^{q-\frac{6p}{p+2}+1}
\mathfrak{r}^{-\frac{2p}{p+2}-1}\nonumber\\
&~~~~+\frac{p}{p+2}\mathfrak{r}^{-\frac{p}{p+2}-1}s^{q-\frac{3p}{p+2}}\left(s^{1-\frac{3p}{p+2}}\mathfrak{r}^{-\frac{p}{p+2}}\right)_{\theta
\theta}.
\nonumber
\end{align}
To apply the maximum principle, we need to bound the right-hand side
of ~(\ref{e: laplacian}) from below.
\begin{align}\label{e: lower bound for laplacian}
\left(s^{1-\frac{3p}{p+2}}\mathfrak{r}^{-\frac{p}{p+2}}\right)_{\theta\theta}&=
\left(s^{q-\frac{3p}{p+2}}\mathfrak{r}^{-\frac{p}{p+2}}s^{1-q}\right)_{\theta\theta}\\
&=s^{1-q}
\left(s^{q-\frac{3p}{p+2}}\mathfrak{r}^{-\frac{p}{p+2}}\right)_{\theta\theta}+s^{q-\frac{3p}{p+2}}\mathfrak{r}^{-\frac{p}{p+2}}
\left(s^{1-q}\right)_{\theta\theta}\nonumber\\&~~~~+2\left(s^{1-q}\right)_{\theta}
\left(s^{q-\frac{3p}{p+2}}\mathfrak{r}^{-\frac{p}{p+2}}\right)_{\theta}\nonumber.
\end{align}
At the point where the minimum of
$s^{q-\frac{3p}{p+2}}\mathfrak{r}^{-\frac{p}{p+2}}$ occurs, we have
$$\left(s^{q-\frac{3p}{p+2}}\mathfrak{r}^{-\frac{p}{p+2}}\right)_{\theta\theta}\ge
0,$$
and
$$\left(s^{q-\frac{3p}{p+2}}\mathfrak{r}^{-\frac{p}{p+2}}\right)_{\theta}=0.$$
Therefore, by equation ~(\ref{e: lower bound for laplacian}), we
obtain that, at that point,
\begin{align}\label{ie: chain}
\left(s^{1-\frac{3p}{p+2}}\mathfrak{r}^{-\frac{p}{p+2}}\right)_{\theta\theta}
&\ge s^{q-\frac{3p}{p+2}}\mathfrak{r}^{-\frac{p}{p+2}}\left(s^{1-q}\right)_{\theta\theta}\nonumber\\
&=s^{q-\frac{3p}{p+2}}\mathfrak{r}^{-\frac{p}{p+2}}\left[\left(1-q\right)s^{-q}
s_{\theta\theta}-\left(1-q\right)q\left(s^{-1-q}\right)s_{\theta}^2\right]\nonumber\\
&\ge
s^{q-\frac{3p}{p+2}}\mathfrak{r}^{-\frac{p}{p+2}}\left[\left(1-q\right)s^{-q}s_{\theta\theta}\right]\\ \nonumber
&=s^{q-\frac{3p}{p+2}}\mathfrak{r}^{-\frac{p}{p+2}}\left[\left(1-q\right)s^{-q}(\mathfrak{r}-s)\right]\nonumber\\
&=(1-q)s^{-\frac{3p}{p+2}}\mathfrak{r}^{-\frac{p}{p+2}+1}-(1-q)s^{-\frac{3p}{p+2}+1}\mathfrak{r}^{-\frac{p}{p+2}}\nonumber,
\end{align}
where, to pass from the second to the third line, we assumed that either
$q=0$ or $q\ge 1.$ Combining ~(\ref{e: laplacian}),~(\ref{e: lower
bound for laplacian}) and ~(\ref{ie: chain}), at the point where the
minimum of $s^{1-\frac{3p}{p+2}}\mathfrak{r}^{-\frac{p}{p+2}}$
occurs, we have
\begin{align*}
\frac{\partial}{\partial t}\left(s^{q-\frac{3p}{p+2}}\mathfrak{r}^{-\frac{p}{p+2}}\right)&\ge \left(\frac{3p}{p+2}-q\right)s^{q-\frac{6p}{p+2}}\mathfrak{r}^{-\frac{2p}{p+2}}+\frac{p}{p+2}s^{q-\frac{6p}{p+2}+1}
\mathfrak{r}^{-\frac{2p}{p+2}-1}\\
&~~~~+\frac{p(1-q)}{p+2}s^{q-\frac{6p}{p+2}}\mathfrak{r}^{-\frac{2p}{p+2}}-\frac{p(1-q)}{p+2}s^{q-\frac{6p}{p+2}+1}
\mathfrak{r}^{-\frac{2p}{p+2}-1}\\
&=\left(\frac{3p}{p+2}-q+\frac{p(1-q)}{p+2}\right)s^{q-
\frac{6p}{p+2}}\mathfrak{r}^{-\frac{2p}{p+2}}+\frac{pq}{p+2}s^{q-\frac{6p}{p+2}+1}\mathfrak{r}^{-\frac{2p}{p+2}-1}.
\end{align*}
Since
$$\frac{3p}{p+2}-q+\frac{p(1-q)}{p+2}$$
is non-negative for
$q\leq\frac{2p}{p+1} $, the claim follows.
\end{proof}
Consequently, we have that
\begin{corollary}\label{cor:Convexity is preserved}
 Convexity of the evolving curves is preserved as long as the flow exists.
\end{corollary}
\begin{proof}
By Proposition ~\ref{cor: strengthned speed}, setting $q=0$, we have that,
as long as the flow exists,
$$\min_{\theta\in\mathbb{S}^1}\left(\frac{\kappa}{s^3}\right)(\theta,t)\ge
\min_{\theta\in\mathbb{S}^1}\left(\frac{\kappa}{s^3}\right)(\theta,0).$$
This inequality implies
$$\kappa(\theta, t)\ge s^{3}(\theta, t)\min_{\theta\in\mathbb{S}^1}\left(\frac{\kappa}{s^3}\right)(\theta,0)> 0 ,$$ which is precisely the claim of the corollary.
\end{proof}
Following an idea from \cite{Tso}, we consider the evolution of
the function
$\frac{s^{1-\frac{3p}{p+2}}\mathfrak{r}^{-\frac{p}{p+2}}}{s-\rho}$, for some appropriate $\rho$, to
obtain an upper bound on the speed of the flow as long as the
inradius of the evolving curve is uniformly bounded from below.
\begin{lemma}\label{lem: upper bound for speed}
If there exists an $ r>0$ such that $s\ge r$ on $[0,T)$, then
$\kappa$ is uniformly bounded from above on $[0,T)$.
\end{lemma}
\begin{proof}
Define $\Psi (x,t):=
\frac{s^{1-\frac{3p}{p+2}}\mathfrak{r}^{-\frac{p}{p+2}}}{s-\rho}$,
where $\rho=\frac{1}{2}r$. For convenience, we set
$\alpha:={1-\frac{3p}{p+2}}$ and $\beta:={-\frac{p}{p+2}}$ . At the
point where the maximum of $\Psi$ occurs, we have
$$\Psi_{\theta}=0, ~~ \Psi_{\theta\theta}\leq 0,$$
hence we obtain
 \begin{equation}\label{ie: akhar}
\left(s^{\alpha}\mathfrak{r}^{\beta}\right)_{\theta\theta}+s^{\alpha}\mathfrak{r}^{\beta}\leq-\frac{\rho s^{\alpha}\mathfrak{r}^{\beta}-s^{\alpha}\mathfrak{r}^{\beta+1}}{s-\rho}.
 \end{equation}
Calculating
$$\frac{\partial \Psi}{\partial t}=
 \frac{s^{\alpha}}{s-\rho} \frac{\partial
\mathfrak{r}^{\beta}}{\partial
t}+\frac{\mathfrak{r}^{\beta}}{s-\rho} \frac{\partial
s^{\alpha}}{\partial t} -\frac{s^{\alpha} \mathfrak{r}^{\beta}}{(s-
\rho)^2} \frac{\partial s}{\partial t},
$$
and using equation (\ref{e: evolution equation}), and  inequality (\ref{ie: akhar}), we
infer that, at the point where the maximum of $\Psi$ is reached, we
have
$$0\leq\frac{\partial}{\partial t}\Psi\leq\frac{1}{s-\rho}\left[-{\beta}s^{\alpha}\mathfrak{r}^{\beta-1}\left(-\frac{\rho s^{\alpha}\mathfrak{r}^{\beta}- s^{\alpha}\mathfrak{r}^{\beta+1}}{s-\rho}\right)+\mathfrak{r}^{\beta}\frac{\partial}{\partial t} s^{\alpha}+ \frac{(s^{2\alpha}\mathfrak{r}^{2\beta})}{s-\rho}\right].$$
This last inequality gives
$$\beta\rho\kappa-\beta-\alpha+\alpha\rho\frac{1}{s}+1\ge 0.$$ Neglecting the non-positive term $\displaystyle \alpha\rho\frac{1}{s}$, we obtain
$$\beta\rho\kappa-\beta-\alpha+1\ge 0.$$ Note that $\displaystyle
\alpha+\beta-1=-\frac{4p}{p+2}$, therefore
 $\displaystyle  0\le\kappa\leq \frac{4}{\rho}$, consequently, implying the lemma.
\end{proof}

\begin{lemma}\label{lem: volume goes to zero}
Let $T$ be the maximal time  of existence of the solution to the
flow (\ref{e: asli}) with a fixed initial body $K_0 \in
\mathcal{K}_{sym}$. Then $T$ is finite and the area of $K_t$,
$A(t)$, tends to zero as $t$ approaches $T$.
\end{lemma}
\begin{proof}
Suppose that $S_0$ is a circle which, at time zero, encloses $K_0$.
It is clear that, by applying the $p$-flow to $S_0$, the evolving circles $S_t$
converge to a point in finite time. By Proposition ~\ref{prop:
containment principle}, $ K_t$ remains in the closure of $S_t$, therefore $T$
must be finite. Suppose now that $A(t)$ does not tend to zero. Then,
we must have $s\geq r$, for some $r>0$ on $[0,T)$. By Corollary
~\ref{cor:Convexity is preserved}, and Lemma \ref{lem: upper bound for
speed}, the curvature of the solution remains bounded on [0,T) from below and above.
Consequently the evolution equation (\ref{e: asli}) is uniformly parabolic on [0,T),
and bounds on higher derivatives of the support function follows by
\cite{K} and Schauder theory. Hence, we can extend the solution after time $T$,  contradicting its definition.
\end{proof}

\begin{lemma}\label{lem:decreasing p-affine surface area}
Assume $1\leq l< 2$. Then, any solution of the flow (\ref{e: asli})  satisfies $\lim_{t\to T}\Omega_l(t)=0$.
\end{lemma}
\begin{proof} From the
$p$-affine isoperimetric inequality in $\mathbb{R}^2$,  \cite{Lutwak},
we have $$0 \leq \Omega_l^{2+l}(t)\leq 2^{2+l}\pi^{2l}A^{2-l}
(t),$$ for any $l\ge 1.$ Therefore, the result is a direct consequence of Lemma \ref{lem: volume goes to
zero}.
\end{proof}
\begin{proposition}\label{prop: length}
 Let $L(t)$ be the length of $\partial K_t$ as $K_t$ is evolving under (\ref{e: asli}). If $ p\geq 1$, then $\lim_{t\to T}L(t)=0.$
\end{proposition}
\begin{proof}
We first seek an $l$ with the following simultaneous properties:
\begin{enumerate}
  \item $ 1\leq l<2,$
  \item $ 1\leq\frac{p}{p+2}\frac{l+2}{l}\leq \frac{2p}{p+1}.$
\end{enumerate}

Note that, by Lemma  ~\ref{lem:decreasing p-affine surface area}, the condition
$(1)$ implies $\lim_{t\to T}\Omega_{l}(t)=0$. The condition $(2)$
implies that
$$\min_{\theta\in\mathbb{S}^1}\left(s\left(\frac{\kappa}{s^3}\right)^{\frac{l}{l+2}}\right)(\theta,t)$$
is increasing. Indeed
$$\left(\min_{\theta\in\mathbb{S}^1}s\left(\frac{\kappa}{s^3}\right)^{\frac{l}{l+2}}(\theta,t)\right)^{\frac{p}{p+2}\frac{l+2}{l}}
=\min_{\theta\in\mathbb{S}^1}\left(s^{\frac{p}{p+2}\frac{l+2}{l}}\left(\frac{\kappa}{s^3}\right)^{\frac{p}{p+2}}\right)(\theta,t),$$
therefore the claim follows from  Proposition \ref{cor: strengthned speed}.
\\ We now proceed to prove the existence of such an $l.$ Solving
$\frac{p}{p+2}\frac{l+2}{l}\leq \frac{2p}{p+1}$ implies $l\geq
\frac{2p+2}{p+3}.$  Let $$l:=\frac{2p+2}{p+3}$$ and note that it
satisfies both conditions $(1)$ and $(2)$.
\\ We further remark that
 \begin{equation}\label{ie: trick}
\min_{\theta\in\mathbb{S}^1}\left(s\left(\frac{\kappa}{s^3}\right)^{\frac{l}{2+l}}\right)(\theta,t)\int_{\mathbb{S}^1}\frac{1}{\kappa}\, d\theta\leq
 \Omega_{l}(t)=\int_{\mathbb{S}^1}\frac{s}{\kappa}\left(\frac{\kappa}{s^3}\right)^{\frac{l}{2+l}}\, d\theta,
 \end{equation}
Thus, by taking the limit as $t \to T$ on both sides of inequality ~(\ref{ie: trick}), we obtain
$$\lim_{t\to T}L(t)=\lim_{t\to T}\int_{\mathbb{S}^1}\frac{1}{\kappa}d\theta=0.$$
\end{proof}

\begin{proposition}\label{prop: volume product}
Centered ellipses are the only homothetic solutions to (\ref{e: asli}).
\end{proposition}
\begin{proof}
Denote by  $A^{\circ}(t):=A(K^{\circ}_t)$ and observe that $A(t)A^{\circ}(t)$  is
scale-invariant. Therefore for homothetic solutions this
area product remains constant along the flow. Moreover, Proposition 2.2 in \cite{S} states, in a larger generality, that, as long as the flow exists,  the $p$-affine flow does not decrease the area product  $A(t)A^{\circ}(t)$ and it remains constant if and only if the evolving curves are ellipses centered at the origin. The result follows now from the existence of solutions until the extinction time of  evolving convex bodies which are centrally symmetric with the center of symmetry placed at the origin.

Alternately, one can argue that having a homothetic solution to (\ref{e: asli}) is equivalent to $\displaystyle \frac{\kappa}{s^3}$ being constant along the boundary of $K_t$. Then, Petty's lemma, \cite{petty}, shows that the latter is equivalent to $K_t$ being an ellipse
 centered at the origin.
\end{proof}

\section{Affine differential setting}

In what follows, we work in the affine setting  to obtain  a sharp affine isoperimetric inequality along the $p$-flow, Theorem \ref{thm: strong affine isoperimetric inequality}. We will now recall several definitions from affine differential geometry. Let $\gamma:\mathbb{S}^1\to\mathbb{R}^2$ be an embedded strictly convex curve with the curve parameter $\theta$.  Define $\mathfrak{g}(\theta):=[\gamma_{\theta},\gamma_{\theta\theta}]^{1/3}$, where, for two vectors $u, v$ in $\mathbb{R}^2$, $[u, v]$ denotes the determinant of the matrix with rows $u$ and $v$.  The
affine arc-length is then given by
\begin{equation}\label{def: affine arclength}
\mathfrak{s}(\theta):=\int_{0}^{\theta}\mathfrak{g}(\xi)d\xi.
\end{equation}
Furthermore, the affine tangent vector $\mathfrak{t}$, the affine normal vector $\mathfrak{n}$, and the affine curvature are defined, in this order, as follows:
\begin{equation*}
\mathfrak{t}:=\gamma_{\mathfrak{s}},~~~ \mathfrak{n}:=\gamma_{\mathfrak{s}\mathfrak{s}}, ~~~\mu:=[\gamma_{\mathfrak{s}\mathfrak{s}}, \gamma_{\mathfrak{s}\mathfrak{s}\mathfrak{s}}].
\end{equation*}
In the affine coordinate ${\mathfrak{s}}$, the following relations hold:
\begin{align}\label{e: some prop of affine setting}
[\gamma_{\mathfrak{s}},\gamma_{\mathfrak{s}\mathfrak{s}}]&=1 \nonumber\\
[\gamma_{\mathfrak{s}},\gamma_{\mathfrak{s}\mathfrak{s}\mathfrak{s}}]&=0\\
[\gamma_{\mathfrak{s}\mathfrak{s}\mathfrak{s}\mathfrak{s}},\gamma_{\mathfrak{s}}]&=\mu\nonumber.
\end{align}
Moreover, it can be easily verified that $\displaystyle  \kappa_0= \frac{[\gamma_{\theta},\gamma_{\theta\theta}]}{[\gamma,\gamma_{\theta} ]^3}=\frac{[\gamma_{\mathfrak{s}},\gamma_{\mathfrak{s}\mathfrak{s}}]}{[\gamma,\gamma_{\mathfrak{s}} ]^3}.$
Since $[\gamma_{\mathfrak{s}},\gamma_{\mathfrak{s}\mathfrak{s}}]=1$, we conclude that $ \displaystyle \kappa_0=\frac{1}{[\gamma,\gamma_{\mathfrak{s}} ]^3}.$ The affine support function is defined by $\sigma:=\kappa_0^{-1/3}$, see \cite{BA3,NS}.
\\ Let $K_0\in \mathcal{K}_{sym}$. We consider a family $\{K_t\}_t\in \mathcal{K}_{sym}$, and their  associated smooth
 embeddings $x:\mathbb{S}^1\times[0,T)\to \mathbb{R}^2$, which are evolving according  to
 \begin{equation}\label{e: affine def of flow}
 \frac{\partial}{\partial t}x:=\sigma^{1-\frac{3p}{p+2}}\mathfrak{n},~~
 x(.,0)=x_{K_0}(.),~~ x(. ,t)=x_{K_t}(.)
 \end{equation}
 for a fixed $p\geq 1$.
Observe that up to  diffeomorphisms the flow defined in  (\ref{e: affine def of flow}) is equivalent to the flow defined by (\ref{e: faree}).
\begin{lemma} Let $\gamma(t):=\partial K_t$ be the boundary of a convex body $K_t$ evolving under the flow (\ref{e: affine def of flow}). Then the following evolution equations hold:

\begin{enumerate}
  \item $\displaystyle \frac{\partial}{\partial t}\mathfrak{g}=-\frac{2}{3}\mathfrak{g}\sigma^{1-\frac{3p}{p+2}}\mu+\frac{1}{3}\mathfrak{g} \left(\sigma^{1-\frac{3p}{p+2}}\right)_{\mathfrak{s}\mathfrak{s}},$
  \item $ \displaystyle \frac{\partial}{\partial t}\mathfrak{t}=\left[-\frac{1}{3}\sigma^{1-\frac{3p}{p+2}}\mu-\frac{1}{3}\left(\sigma^{1-\frac{3p}{p+2}}
      \right)_{\mathfrak{s}\mathfrak{s}}\right]\mathfrak{t}+
\left(\sigma^{1-\frac{3p}{p+2}}\right)_{\mathfrak{s}}\mathfrak{n},$
  \item $\displaystyle \frac{\partial}{\partial t}\sigma=\sigma^{1-\frac{3p}{p+2}}\left[-\frac{4}{3}+
\left(\frac{p}{p+2}+1\right)\left(1-\frac{3p}{p+2}\right)\frac{\sigma_{\mathfrak{s}}^2}{\sigma}+\frac{p}{p+2}\sigma_{\mathfrak{s}\mathfrak{s}}
\right],$
  \item $ \displaystyle \frac{d}{dt}\Omega_p (t)=\frac{2(p-2)}{p+2}\int_{\gamma} \sigma^{1-\frac{6p}{p+2}}d\mathfrak{s}+\frac{18p^2}{(p+2)^3}\int_{\gamma} \sigma^{-\frac{6p}{p+2}}\sigma_{\mathfrak{s}}^2d\mathfrak{s}.$
\end{enumerate}
\label{lemma:4parts}
\end{lemma}
\begin{proof}
To prove the lemma we will use repeatedly equations (\ref{e: some prop of affine setting}) without further mention.
\begin{align*}
\frac{\partial}{\partial t}\mathfrak{g}^3&=\frac{\partial}{\partial t}[\gamma_{\theta},\gamma_{\theta\theta}]=[\frac{\partial}{\partial t}\gamma_{\theta},\gamma_{\theta\theta}]+[\gamma_{\theta},\frac{\partial}{\partial t}\gamma_{\theta\theta}].\\
\end{align*}
We have that
\begin{align*}
\left[\frac{\partial}{\partial t}\gamma_{\theta},\gamma_{\theta\theta}\right]&=\left[\frac{\partial}{\partial \theta}\left(\sigma^{1-\frac{3p}{p+2}}\gamma_{\mathfrak{s}\mathfrak{s}}\right),\gamma_{\theta\theta}\right]\\
&=\left[\mathfrak{g}\frac{\partial}{\partial \mathfrak{s}}\left(\sigma^{1-\frac{3p}{p+2}}\gamma_{\mathfrak{s}\mathfrak{s}}\right),\gamma_{\theta\theta}\right]\\
&=\mathfrak{g}\left[\left(\sigma^{1-\frac{3p}{p+2}}\right)_{\mathfrak{s}}\gamma_{\mathfrak{s}\mathfrak{s}}+
\sigma^{1-\frac{3p}{p+2}}\gamma_{\mathfrak{s}\mathfrak{s}\mathfrak{s}},\gamma_{\theta\theta}\right].
\end{align*}
Since
$\displaystyle \frac{\partial^2}{\partial\theta^2}=\mathfrak{g}\mathfrak{g}_{\mathfrak{s}}\frac{\partial}{\partial \mathfrak{s}}
+\mathfrak{g}^2\frac{\partial^2}{\partial\mathfrak{s}^2}$, we further have
 $\gamma_{\theta\theta}=\mathfrak{g}^2\gamma_{\mathfrak{s}\mathfrak{s}}+\mathfrak{g}\mathfrak{g}_{\mathfrak{s}}\gamma_{\mathfrak{s}}$
and, therefore,
\begin{align*}
\left[\frac{\partial}{\partial t}\gamma_{\theta},\gamma_{\theta\theta}\right]
&=\mathfrak{g}\left[\left(\sigma^{1-\frac{3p}{p+2}}\right)_{\mathfrak{s}}\gamma_{\mathfrak{s}\mathfrak{s}}+
\sigma^{1-\frac{3p}{p+2}}\gamma_{\mathfrak{s}\mathfrak{s}\mathfrak{s}},\mathfrak{g}^2\gamma_{\mathfrak{s}\mathfrak{s}}
+\mathfrak{g}\mathfrak{g}_{\mathfrak{s}}\gamma_{\mathfrak{s}}\right]\\
&=-\mathfrak{g}^2\mathfrak{g}_{\mathfrak{s}}\left(\sigma^{1-\frac{3p}{p+2}}\right)_{\mathfrak{s}}-
\mathfrak{g}^3\sigma^{1-\frac{3p}{p+2}}\mu.
\end{align*}
On the other hand, we have
\begin{align*}
\left[\gamma_{\theta},\frac{\partial}{\partial t}\gamma_{\theta\theta}\right]&=\left[\mathfrak{g}\gamma_{\mathfrak{s}},
\frac{\partial^2}{\partial\theta^2}\left(\sigma^{1-\frac{3p}{p+2}}\gamma_{\mathfrak{s}\mathfrak{s}}\right)\right]\\
&=\left[\mathfrak{g}\gamma_{\mathfrak{s}},
\mathfrak{g}\mathfrak{g}_{\mathfrak{s}}\frac{\partial}{\partial \mathfrak{s}}\left(\sigma^{1-\frac{3p}{p+2}}\gamma_{\mathfrak{s}\mathfrak{s}}\right)+
\mathfrak{g}^2\frac{\partial^2}{\partial\mathfrak{s}^2}\left(\sigma^{1-\frac{3p}{p+2}}\gamma_{\mathfrak{s}\mathfrak{s}}\right)\right]\\
&=\mathfrak{g}^2\mathfrak{g}_{\mathfrak{s}}\left(\sigma^{1-\frac{3p}{p+2}}\right)_{\mathfrak{s}}+
\mathfrak{g}^3\left(\sigma^{1-\frac{3p}{p+2}}\right)_{\mathfrak{s}\mathfrak{s}}-\mathfrak{g}^3\sigma^{1-\frac{3p}{p+2}}\mu.
\end{align*}
Hence, we conclude that
$$\frac{\partial}{\partial t}\mathfrak{g}^3=\mathfrak{g}^3\left(\sigma^{1-\frac{3p}{p+2}}\right)_{\mathfrak{s}\mathfrak{s}}-
2\mathfrak{g}^3\sigma^{1-\frac{3p}{p+2}}\mu,$$
which verifies our first claim.\\
To prove the second claim, we observe that
\begin{equation}\label{e: commute}
\frac{\partial}{\partial t}\frac{\partial}{\partial \mathfrak{s}}=\frac{\partial}{\partial\mathfrak{s}}\frac{\partial}{\partial t}-
\frac{1}{\mathfrak{g}}\frac{\partial\mathfrak{g}}{\partial t}\frac{\partial}{\partial \mathfrak{s}}.
\end{equation}
By (\ref{e: commute}), we get
\begin{align*}
\frac{\partial}{\partial t}\mathfrak{t}&=\frac{\partial}{\partial t}\frac{\partial}{\partial \mathfrak{s}}\gamma\\
&=\frac{\partial}{\partial \mathfrak{s}}\left(\sigma^{1-\frac{3p}{p+2}}\gamma_{\mathfrak{s}\mathfrak{s}}\right)+\left(\frac{2}{3}\sigma^{1-\frac{3p}{p+2}}\mu-\frac{1}{3} \left(\sigma^{1-\frac{3p}{p+2}}\right)_{\mathfrak{s}\mathfrak{s}}\right)\mathfrak{t}\\
&=\left(\sigma^{1-\frac{3p}{p+2}}\right)_{\mathfrak{s}}\mathfrak{n}+\sigma^{1-\frac{3p}{p+2}}\gamma_{\mathfrak{s}\mathfrak{s}\mathfrak{s}}
+\left(\frac{2}{3}\sigma^{1-\frac{3p}{p+2}}\mu-\frac{1}{3} \left(\sigma^{1-\frac{3p}{p+2}}\right)_{\mathfrak{s}\mathfrak{s}}\right)\mathfrak{t}.
\end{align*}
We note that $\gamma_{\mathfrak{s}\mathfrak{s}\mathfrak{s}}=-\mu \gamma_{\mathfrak{s}}$ ending the proof of the second claim.\\
We now proceed to prove the third claim with
\begin{align*}
\frac{\partial}{\partial t}\sigma=\frac{\partial}{\partial t}\left[\gamma,\gamma_{\mathfrak{s}}\right]=\left[\frac{\partial}{\partial t}\gamma, \gamma_{\mathfrak{s}}\right]+\left[\gamma, \frac{\partial}{\partial t}\gamma_{\mathfrak{s}} \right].
\end{align*}
By the evolution equation (\ref{e: affine def of flow}), the evolution equation for $\mathfrak{t}$, and the identities $\sigma=[\gamma,\gamma_{\mathfrak{s}}]$ and $\sigma_{\mathfrak{s}}=[\gamma,\gamma_{\mathfrak{s}\mathfrak{s}}]$, we get that
\begin{align*}
\frac{\partial}{\partial t}\sigma&=\left[\sigma^{1-\frac{3p}{p+2}}\gamma_{\mathfrak{s}\mathfrak{s}}, \gamma_{\mathfrak{s}}\right]+\left[\gamma, \left(-\frac{1}{3}\sigma^{1-\frac{3p}{p+2}}\mu-
\frac{1}{3}\left(\sigma^{1-\frac{3p}{p+2}}\right)_{\mathfrak{s}\mathfrak{s}}\right)\gamma_\mathfrak{s}
+\left(\sigma^{1-\frac{3p}{p+2}}\right)_{\mathfrak{s}}\gamma_{\mathfrak{s}\mathfrak{s}}\right]\\
&=-\sigma^{1-\frac{3p}{p+2}}-\frac{1}{3}\sigma^{2-\frac{3p}{p+2}}\mu-
\frac{1}{3}\left(\sigma^{1-\frac{3p}{p+2}}\right)_{\mathfrak{s}\mathfrak{s}}\sigma+
\left(\sigma^{1-\frac{3p}{p+2}}\right)_{\mathfrak{s}}\sigma_{\mathfrak{s}}\\
&=-\sigma^{1-\frac{3p}{p+2}}-\frac{1}{3}\sigma^{2-\frac{3p}{p+2}}\mu+
\left(-\frac{1}{3}+\frac{p}{p+2}\right)\sigma^{1-\frac{3p}{p+2}}\sigma_{\mathfrak{s}\mathfrak{s}}\\
&~~~~+\left(1-\frac{3p}{p+2}\right)\frac{p}{p+2}\sigma^{-\frac{3p}{p+2}}\sigma_{\mathfrak{s}}^2+
\left(\sigma^{1-\frac{3p}{p+2}}\right)_{\mathfrak{s}}\sigma_{\mathfrak{s}}.
\end{align*}
Observe that $\sigma_{\mathfrak{s}\mathfrak{s}}+\sigma\mu=1$, and apply it to the second and third term of last sum, to obtain
\begin{align*}
\frac{\partial}{\partial t}\sigma&=-\frac{4}{3}\sigma^{1-\frac{3p}{p+2}}+
\frac{p}{p+2}\sigma^{1-\frac{3p}{p+2}}\sigma_{\mathfrak{s}\mathfrak{s}}+
\frac{p}{p+2}\left(1-\frac{3p}{p+2}\right)\sigma^{-\frac{3p}{p+2}}\sigma_{\mathfrak{s}}^2+
\left(\sigma^{1-\frac{3p}{p+2}}\right)_{\mathfrak{s}}\sigma_{\mathfrak{s}}\\
&=-\frac{4}{3}\sigma^{1-\frac{3p}{p+2}}+\left(\frac{p}{p+2}+1\right)
\left(1-\frac{3p}{p+2}\right)\sigma^{-\frac{3p}{p+2}}\sigma_{\mathfrak{s}}^2
+\frac{p}{p+2}\left(\sigma^{1-\frac{3p}{p+2}}\right)\sigma_{\mathfrak{s}\mathfrak{s}},
\end{align*}
as claimed.\\
For the last claim of the lemma, consider
\begin{align*}
\frac{d}{d t}\Omega_p (t)&=\frac{\partial}{\partial t}\int_{\gamma} \sigma^{1-\frac{3p}{p+2}}d\mathfrak{s}\\
&=\int_{\gamma} \frac{\partial}{\partial t}\left(\sigma^{1-\frac{3p}{p+2}}\right)d\mathfrak{s}+\int_{\gamma} \sigma^{1-\frac{3p}{p+2}}\frac{\partial}{\partial t}d\mathfrak{s}.
\end{align*}
Using the previous part $(3)$ of the lemma, and integration by parts, we obtain
\begin{align*}
\int_{\gamma} \frac{\partial}{\partial t}\left(\sigma^{1-\frac{3p}{p+2}}\right)d\mathfrak{s}&=
\left(\frac{4p}{p+2}-\frac{4}{3}\right)\int_{\gamma} \sigma^{1-\frac{6p}{p+2}}d\mathfrak{s}+\left(\frac{p}{p+2}+1\right)
\left(1-\frac{3p}{p+2}\right)^2\int_{\gamma}\sigma^{-\frac{6p}{p+2}}\sigma_{\mathfrak{s}}^2d\mathfrak{s}\\
&~~~~+\frac{p}{p+2}\left(1-\frac{3p}{p+2}\right)\int_{\gamma}\sigma^{1-\frac{6p}{p+2}}\sigma_{\mathfrak{s}\mathfrak{s}}d\mathfrak{s}\\
&=\left(\frac{4p}{p+2}-\frac{4}{3}\right)\int_{\gamma} \sigma^{1-\frac{6p}{p+2}}d\mathfrak{s}+\left(\frac{p}{p+2}+1\right)
\left(1-\frac{3p}{p+2}\right)^2\int_{\gamma}\sigma^{-\frac{6p}{p+2}}\sigma_{\mathfrak{s}}^2d\mathfrak{s}\\
&~~~~+\frac{p}{p+2}\left(1-\frac{3p}{p+2}\right)\left(\frac{6p}{p+2}-1\right)\int_{\gamma}\sigma^{-\frac{6p}{p+2}}\sigma_{\mathfrak{s}}^2
d\mathfrak{s}.
\end{align*}
On the other hand, (\ref{def: affine arclength}) gives
 $$d\mathfrak{s}=\mathfrak{g}d\theta$$
thus, by part $(1)$, we get
 $$\frac{\partial}{\partial t}d\mathfrak{s}=\left[-\frac{2}{3}\sigma^{1-\frac{3p}{p+2}}\mu+\frac{1}{3} \left(\sigma^{1-\frac{3p}{p+2}}\right)_{\mathfrak{s}\mathfrak{s}}\right]d\mathfrak{s}.$$
 This implies
 \begin{align*}
 \int_{\gamma}\sigma^{1-\frac{3p}{p+2}}\frac{\partial}{\partial t}d\mathfrak{s}&=\int_{\gamma} \sigma^{1-\frac{3p}{p+2}}\left[-\frac{2}{3}\sigma^{1-\frac{3p}{p+2}}\mu+\frac{1}{3} \left(\sigma^{1-\frac{3p}{p+2}}\right)_{\mathfrak{s}\mathfrak{s}}\right]d\mathfrak{s}\\
 &=-\frac{2}{3}\int_{\gamma} \sigma^{1-\frac{6p}{p+2}}(1-\sigma_{\mathfrak{s}\mathfrak{s}})d\mathfrak{s}-
 \frac{1}{3}\left({1-\frac{3p}{p+2}}\right)^2 \int_{\gamma} \sigma^{-\frac{6p}{p+2}}\sigma_{\mathfrak{s}}^2d\mathfrak{s}\\
 &=-\frac{2}{3}\int_{\gamma} \sigma^{1-\frac{6p}{p+2}}d\mathfrak{s}+\left[\frac{2}{3}\left(\frac{6p}{p+2}-1\right)
 -\frac{1}{3}\left({1-\frac{3p}{p+2}}\right)^2\right] \int_{\gamma} \sigma^{-\frac{6p}{p+2}}\sigma_{\mathfrak{s}}^2d\mathfrak{s}.
 \end{align*}
 Setting
 \begin{align*}
 Q:&=\left(\frac{p}{p+2}+1\right)
\left(1-\frac{3p}{p+2}\right)^2-\frac{p}{p+2}\left(\frac{3p}{p+2}-1\right)\left(\frac{6p}{p+2}-1\right)
\\ &~~~~+\frac{2}{3}\left(\frac{6p}{p+2}-1\right)
 -\frac{1}{3}\left({1-\frac{3p}{p+2}}\right)^2=\frac{18p^2}{(p+2)^3},
\end{align*}
 and combining the above equations, we finally acquire that
 \begin{align*}
 \frac{d}{dt}\Omega_p (t)&=\frac{2(p-2)}{p+2}\int_{\gamma} \sigma^{1-\frac{6p}{p+2}}d\mathfrak{s}+Q\int_{\gamma} \sigma^{-\frac{6p}{p+2}}\sigma_{\mathfrak{s}}^2d\mathfrak{s}.
 \end{align*}
\end{proof}
Now, we proceed to strengthen inequality ~(\ref{ie: alina}). Let
$K$ and $L$ be two convex bodies with support functions $s$ and $h$,
respectively. Then, the mixed volume of $K$ and $L$ is defined by
$$V[s,h]=\int_{\mathbb{S}^1}s\mathfrak{r}[h]d\theta=
\int_{\mathbb{S}^1}h\mathfrak{r}[s]d\theta.$$

By Minkowski's mixed volume
inequality \cite{schneider}, we have
\begin{equation}\label{ie: original mixed volume}
V^2[h,s]\ge V[s,s]V[h,h].
\end{equation}
More interestingly, inequality (\ref{ie: original mixed volume}) still holds if $h$ is an arbitrary  function in $ \mathcal{C}^2(\mathbb{S}^1)$. Indeed, assuming that $h$ is not the support function of some convex body, for a large positive constant $c$, the sum $h+cs$ is a support function and we obtain,  due to the linearity of mixed volumes,
$$0 \leq V^2[h+cs,s]- V[h+cs,h+cs]V[s,s]=V^2[h,s]- V[h,h]V[s,s].$$
The following Proposition, stated here only for $n=2$, is proved in \cite{S} for all dimensions.
Using our method in this section, we prove a stronger version of the planar inequality in Theorem \ref{thm: strong affine isoperimetric inequality}.
\begin{proposition}\label{thm: second affine inequality}
Let  $p\ge 1$, as $K_t$ evolves under (\ref{e: asli}).
Then we have
\begin{equation} \label{ie: alina}
\frac{d}{dt}\Omega_p(t)\geq \frac{p-2}{p+2}\frac{\Omega_p^2(t)}{A(t)},
\end{equation}
with equality if and only if $K_t$ is an origin centered ellipse.
\end{proposition}

\begin{theorem}\label{thm: strong affine isoperimetric inequality}
The following strong affine isoperimetric inequalities hold. \\
If $1\leq p\leq 2$, then
\begin{align}\label{e: strong affine isoperimetric inequality1}
\frac{d}{dt}\Omega_p(t)\geq \frac{p-2}{p+2}\frac{\Omega_p^2}{A}+\frac{18(p-1)p^2}{(p+2)^3}\int_{\gamma} \sigma^{-\frac{6p}{p+2}}\sigma_{\mathfrak{s}}^2d\mathfrak{s},
\end{align}
while, if $p\ge 2$, we then have
\begin{align}\label{e: strong affine isoperimetric inequality2}
\frac{d}{dt}\Omega_p(t)\geq  \frac{p-2}{p+2}\frac{\Omega_p^2}{A}+\frac{18p^2}{(p+2)^3}\int_{\gamma} \sigma^{-\frac{6p}{p+2}}\sigma_{\mathfrak{s}}^2d\mathfrak{s}.
\end{align}
\end{theorem}
\begin{proof}
To prove the second statement, we note that H\"{o}lder's inequality gives
 $$\int_{\gamma} \sigma^{1-\frac{6p}{p+2}}d\mathfrak{s}=\int_{\mathbb{S}^1}\frac{s}{\kappa}\left(\frac{\kappa}{s^3}\right)^{\frac{2p}{p+2}}d\theta\ge \frac{\Omega_p^2}{2A}$$
 and, thus, part $(4)$ of Lemma \ref{lemma:4parts} implies the affine isoperimetric inequality for $p\geq2.$
 We now proceed to prove the first inequality. By Minkowski's mixed volume inequality (\ref{ie: original mixed volume}), we have
 \begin{align}\label{ie: mixed 1}
 V\left[s\left(\frac{\kappa}{s^3}\right)^{\frac{p}{p+2}}, s\left(\frac{\kappa}{s^3}\right)^{\frac{p}{p+2}}\right]\leq
 \frac{ V\left[s\left(\frac{\kappa}{s^3}\right)^{\frac{p}{p+2}}, s\right]^2}{V[s,s]}.
 \end{align}
 Note that the right hand side of the inequality is precisely $\displaystyle \frac{\Omega_p^2}{2A}$. Using the identity, see \cite{BA3},
 \begin{align*}
 \frac{\mathfrak{r}\left[s\left(\frac{\kappa}{s^3}\right)^{\frac{p}{p+2}}\right]}{\mathfrak{r}[s]}
 &=\frac{s}{\kappa^{\frac{1}{3}}}\left(\frac{\kappa}{s^3}\right)^{\frac{p}{p+2}}\mu
 +\left(\frac{s}{\kappa^{\frac{1}{3}}}\left(\frac{\kappa}{s^3}\right)^{\frac{p}{p+2}}\right)_{\mathfrak{s}\mathfrak{s}}\\
 &=\sigma^{1-\frac{3p}{p+2}}\mu+\left(\sigma^{1-\frac{3p}{p+2}}\right)_{\mathfrak{s}\mathfrak{s}},
 \end{align*}
 we can rewrite the left hand side of  (\ref{ie: mixed 1}) as follows:
 \begin{align}\label{e: rewrite}
  V\left[s\left(\frac{\kappa}{s^3}\right)^{\frac{p}{p+2}}, s\left(\frac{\kappa}{s^3}\right)^{\frac{p}{p+2}}\right]&=\int_{\mathbb{S}^1}s\left(\frac{\kappa}{s^3}\right)^{\frac{p}{p+2}}
  \mathfrak{r}\left[s\left(\frac{\kappa}{s^3}\right)^{\frac{p}{p+2}}\right]d\theta\nonumber\\
  &=\int_{\gamma}\frac{s}{\kappa^{\frac{1}{3}}}\left(\frac{\kappa}{s^3}\right)^{\frac{p}{p+2}}
  \frac{\mathfrak{r}\left[s\left(\frac{\kappa}{s^3}\right)^{\frac{p}{p+2}}\right]}{\mathfrak{r}}d\mathfrak{s}\nonumber\\
  &=\int_{\gamma}\sigma^{1-\frac{3p}{p+2}}\left(\sigma^{1-\frac{3p}{p+2}}\mu+\left(\sigma^{1-\frac{3p}{p+2}}\right)_{\mathfrak{s}\mathfrak{s}}
  \right)
  d\mathfrak{s}\nonumber\\
  &=\int_{\gamma}\sigma^{2-\frac{6p}{p+2}}\mu d\mathfrak{s}-\left(1-\frac{3p}{p+2}\right)^2\int_{\gamma}\sigma^{-\frac{6p}{p+2}}\sigma_{\mathfrak{s}}^2d\mathfrak{s}.
 \end{align}
 Hence, combining equation (\ref{e: rewrite}) and inequality (\ref{ie: mixed 1}), we conclude that
 \begin{align}\label{ie: mixed 2}
 \int_{\gamma}\sigma^{2-\frac{6p}{p+2}}\mu d\mathfrak{s}\leq\frac{\Omega_p^2}{2A}+\left(1-\frac{3p}{p+2}\right)^2\int_{\gamma}\sigma^{-\frac{6p}{p+2}}\sigma_{\mathfrak{s}}^2d\mathfrak{s}.
 \end{align}
 Inequality (\ref{ie: mixed 2}) is a special case of affine-geometric Wirtinger inequality, Lemma 6, \cite{BA3}.
 To finish the proof note also that
 \begin{align*}
 \frac{d}{dt}\Omega_p(t)&=\frac{2(p-2)}{p+2}\int_{\gamma} \sigma^{1-\frac{6p}{p+2}}d\mathfrak{s}+\frac{18p^2}{(p+2)^3}\int_{\gamma} \sigma^{-\frac{6p}{p+2}}\sigma_{\mathfrak{s}}^2d\mathfrak{s}\\
 &=\frac{2(p-2)}{p+2}\int_{\gamma} \sigma^{2-\frac{6p}{p+2}}\left(\frac{1}{\sigma}-\frac{\sigma_{\mathfrak{s}\mathfrak{s}}}{\sigma}\right)d\mathfrak{s}
 +\frac{2(p-2)}{p+2}\int_{\gamma} \sigma^{2-\frac{6p}{p+2}}\frac{\sigma_{\mathfrak{s}\mathfrak{s}}}{\sigma}d\mathfrak{s}
 +\frac{18p^2}{(p+2)^3}\int_{\gamma} \sigma^{-\frac{6p}{p+2}}\sigma_{\mathfrak{s}}^2d\mathfrak{s}\\
 &=\frac{2(p-2)}{p+2}\int_{\gamma} \sigma^{2-\frac{6p}{p+2}}\mu d\mathfrak{s}+
 \left[\frac{2(p-2)}{p+2}\left(\frac{6p}{p+2}-1\right)+\frac{18p^2}{(p+2)^3}\right]\int_{\gamma} \sigma^{-\frac{6p}{p+2}}\sigma_{\mathfrak{s}}^2d\mathfrak{s},
 \end{align*}
which, by inequality (\ref{ie: mixed 2}), implies
 \begin{align*}
 \frac{d}{dt}\Omega_p(t)&\geq \frac{p-2}{p+2}\frac{\Omega_{p}^2}{A}
 +\left[\frac{2(p-2)}{p+2}\left(1-\frac{3p}{p+2}\right)^2+\frac{2(p-2)}{p+2}\left(\frac{6p}{p+2}-1\right)+\frac{18p^2}{(p+2)^3}\right]
 \int_{\gamma} \sigma^{-\frac{6p}{p+2}}\sigma_{\mathfrak{s}}^2d\mathfrak{s}\\
 &=\frac{p-2}{p+2}\frac{\Omega_{p}^2}{A}+\frac{18(p-1)p^2}{(p+2)^3}\int_{\gamma} \sigma^{-\frac{6p}{p+2}}\sigma_{\mathfrak{s}}^2d\mathfrak{s}.
 \end{align*}
\end{proof}

\begin{lemma}\label{lem: p-ratio}
The $p$-affine isoperimetric ratio,
$\frac{\Omega_p^{2+p}(t)}{A^{2-p}(t)}$, is non-decreasing along the
flow (\ref{e: asli}) and remains constant if and only $K_t$ is an origin centered ellipse.
\end{lemma}
\begin{proof}
\begin{align}\label{e: p-ratio}
\frac{d}{d t}\frac{\Omega_p^{2+p}(t)}{A^{2-p}(t)}&=\frac{(2+p)\Omega_p^{p+1}(t)A^{2-p}(t)\frac{d}{dt}\Omega_p(t)+(2-p)A^{1-p}(t)\Omega_p^{3+p}(t)}{A^{2(2-p)}(t)}
\\ &=\frac{\Omega_p^{1+p}(t)}{A^{2-p}(t)}\left((2+p)\frac{d}{dt}\Omega_p(t)-(p-2)\frac{\Omega_p^2(t)}{A(t)}\right)\ge
0\nonumber,
\end{align}
where we used inequality ~(\ref{ie: alina}) on the last line.
\end{proof}
\begin{corollary}\label{cor: limsup idea} If $K_t$ evolves by (\ref{e: asli}) with extinction time $T$, the following limit holds:
\begin{equation}
 \lim\inf_{t\to T} \frac{\Omega_p^p(t)}{A^{1-p}(t)}\left[(2+p)\frac{d}{dt}\Omega_p(t)-(p-2)\frac{\Omega_p^2(t)}{A(t)}\right]=0.
 \label{eq:sup}
 \end{equation}
\end{corollary}
\begin{proof}
By equations ~(\ref{e: volume}) and ~(\ref{e: p-ratio})
\begin{align*}
\frac{d}{d t}\frac{\Omega_p^{2+p}(t)}{A^{2-p}(t)}
&=-\frac{d}{d t}\ln(A(t))\left(\frac{\Omega_p^p(t)}{A^{1-p}(t)}\left[(2+p)\frac{d}{dt}\Omega_p(t)-(p-2)\frac{\Omega_p^2(t)}{A(t)}\right]\right).
\end{align*}
If $$\frac{\Omega_p^p(t)}{A^{1-p}(t)}\left[(2+p)\frac{d}{dt}\Omega_p(t)-(p-2)\frac{\Omega_p^2(t)}{A(t)}\right]\geq\varepsilon$$ in a neighborhood of $T$, then
$$\frac{d}{dt}\frac{\Omega_p^{2+p}(t)}{A^{2-p}(t)}\ge -\varepsilon \frac{d}{d t}\ln(A(t)).$$
Thus,
$$\frac{\Omega_p^{2+p}}{A^{2-p}}(t)\ge \frac{\Omega_p^{2+p}}{A^{2-p}}(t_1)+\varepsilon \ln(A(t_1))-\varepsilon \ln(A(t)),$$
the right hand side goes to infinity as $A(t)$ goes to zero. This contradicts that the left hand side is bounded from above by the $p$-affine isoperimetric inequality.
\end{proof}
\section{Normalized flow}
In this section we study the normalized flows corresponding to the evolution described by (\ref{e: asli}).
We consider the conventional rescaling such that the area enclosed by the normalized curves is $\pi$ by taking
 $$\tilde{s}_t:=\sqrt{\frac{\pi}{A(t)}}\,s_t,~~~\tilde{\kappa}_t:=\sqrt{\frac{A(t)}{\pi}}\,{\kappa}_t.$$   One can also define a new time parameter $$\tau=\int_{0}^t\left(\frac{\pi}{A(K_t)(\xi)}\right)^{\frac{2p}{p+2}}d\xi$$  and can easily verify that
\begin{equation}\label{e: normalized flow}
\frac{\partial }{\partial \tau}\, \tilde{s}=-\tilde{s}\left(\frac{\tilde{\kappa}}{\tilde{s}^3}\right)^{\frac{p}{p+2}}+
\frac{\tilde{s}}{2\pi}\,\tilde{\Omega}_p,
\end{equation}
where $\tilde{\Omega}_p$ stands for the $p$-affine length of $\partial \tilde{K}_t$ having support function $\tilde{s}_t$. More precisely,
$$\tilde{\Omega}_p(\tau):=\Omega_p(\tilde{K}_{\tau})=
\int_{\mathbb{S}^1}\frac{\tilde{s}}{\tilde{\kappa}}\left(\frac{\tilde{\kappa}}{\tilde{s}^3}\right)^{\frac{p}{p+2}}d\theta.$$
However, even in the normalized case, we prefer to work on the finite time interval $[0,T)$.
\begin{corollary} \label{cor: limit of affine support} Let $p$ be a real number, $p>1$, and let $\{t_k\}_k$ be the sequence of times realizing the limit (\ref{eq:sup}) in Corollary \ref{cor: limsup idea}. Then, along the normalized contracting $p$-flow, we have
$$\lim_{t_k\to T}\tilde{\sigma}(t_k)=1.$$
\end{corollary}
\begin{proof}

Since
$$\frac{1}{\left(\frac{3p}{p+2}-1\right)^2}\int_{\gamma}\left(\sigma^{1-\frac{3p}{p+2}}\right)_{\mathfrak{s}}^2d\mathfrak{s}=\int_{\gamma} \sigma^{-\frac{6p}{p+2}}\sigma_{\mathfrak{s}}^2d\mathfrak{s},$$
by  Theorem \ref{thm: strong affine isoperimetric inequality} and Corollary \ref{cor: limsup idea}, we have
$$0=\lim_{t_k\to T}\frac{\Omega_p^{p}}{A^{1-p}}\left[\frac{d}{dt}\Omega_p(t)-\frac{p-2}{p+2}\frac{\Omega_p^2}{A}\right]\geq \lim_{t_k\to T}\frac{\Omega_p^{p}}{A^{1-p}}\left(\phi(p)\int_{\gamma}\left(\sigma^{1-\frac{3p}{p+2}}\right)_{\mathfrak{s}}^2d\mathfrak{s}\right)\geq 0,$$ where
 $\phi(p):=\left\{
         \begin{array}{ll}
           \frac{9p^2}{2(p+2)(p-1)}, & \hbox{if~} 1<p\leq2 \\
           \frac{9p^2}{2(p+2)(p-1)^2}, & \hbox{if~} p\ge2.
         \end{array}
       \right.$

As, by Lemma \ref{lem: p-ratio}, the $p$-affine length $\tilde{\Omega}_p$ is increasing along the normalized flow, we conclude that, for any $p>1$,
 $$\lim_{t_k\to T}\int_{\tilde{\gamma}}\left(\tilde{\sigma}^{1-\frac{3p}{p+2}}\right)_{\tilde{\mathfrak{s}}}^2d\tilde{\mathfrak{s}}=0.$$
 We note that, for any $\theta_1, \theta_2 \in {\mathbb{S}^1}$,

\hspace{0.5cm} $ \displaystyle \left|\int_{\theta_1}^{\theta_2}\left(\tilde{\sigma}^{1-\frac{3p}{p+2}}\right)_{\theta}\,d\theta \right| \leq $ $$ \int_{\mathbb{S}^1}\left|\left(\tilde{\sigma}^{1-\frac{3p}{p+2}}\right)_{\theta}\right|d\theta =\int_{\tilde{\gamma}}\left|\left(\tilde{\sigma}^{1-\frac{3p}{p+2}}\right)_{\tilde{\mathfrak{s}}}\right|d\tilde{\mathfrak{s}}\leq \left(\int_{\tilde{\gamma}}\left(\tilde{\sigma}^{1-\frac{3p}{p+2}}\right)_{\tilde{\mathfrak{s}}}^2d\tilde{\mathfrak{s}}\right)^{1/2}
\tilde{\Omega}_1^{1/2}. $$
Take $\theta_1$ and $\theta_2$ be two points where $\tilde{\sigma}$ reaches its extremal values. It is known that, for a smooth, simple curve with enclosed area $\pi$, $\min_{\mathbb{S}^1}\sigma \leq 1$ and $\max_{\mathbb{S}^1}\sigma \geq 1$, see Lemma 10 in \cite{BA3}. Hence, as $\tilde{\Omega}_1$ is bounded from above by the classical affine isoperimetric inequality \cite{Lutwak}, we infer that $\lim_{t_k\to T}\tilde{\sigma}(t_k)=1.$
\end{proof}

\begin{lemma} \label{lem: ellipsoid app}\cite{SG} Suppose that $K$ is a convex body in $K_{sym}$. Denote the curvature and the support function of $\partial K$ respectively by $\kappa$ and $s$. If, for all $\displaystyle \theta$: $m\le\frac{\kappa}{s^3}(\theta)\le M$ for some positive numbers $m$ and $M$, then there exist two ellipses $\mathcal{E}_{in}$ and $\mathcal{E}_{out}$ such that $\mathcal{E}_{in}\subseteq K\subseteq \mathcal{E}_{out}$ and $$\frac{\kappa}{s^3}({\mathcal{E}_{in}})=M,~~~ \frac{\kappa}{s^3}(\mathcal{E}_{out})=m .$$
\end{lemma}
\begin{proof}
We present here the argument for the inner ellipse, the case of the outer one being similar.\\
 Recall that
$$\kappa_0=\frac{[\dot{\gamma}, \ddot{\gamma}]}{[\gamma, \dot{\gamma}]^3},$$
where $t\mapsto \gamma(t)$ is any counter-clockwise parametrization
of the boundary curve. For an ellipse, this is a constant inverse
proportional to the square of its area. So, we have to prove that
the maximum-area ellipse contained in $K$ has $\kappa_0\leq M.$ Let
$\mathcal{E}_{in}$ be the maximum-area ellipse contained in $K$.
Since the problem is centro-affine invariant, we may assume that
$\mathcal{E}_{in}$ is the unit circle. We will prove that $M\ge 1.$
The result will then follow by shrinking the circle
$\mathcal{E}_{in}$ until its centro-affine curvature is exactly $M$
and re-denoting it, for simplicity, the same way.
\\ Consider the points where $\partial K$ touches $\mathcal{E}_{in}$ one easily sees that, there are at least four intersection points between $\partial K$ and $\mathcal{E}_{in}$, otherwise $\mathcal{E}_{in}$ could be made larger. Thus, at least two of the intervals on the circle corresponding to the polar angle of the intersection points are not greater than $\pi/2$. In fact, due to the symmetry of $K$, there exist at least two diametrically opposite such intervals. Choose coordinates so that one of the intersection points is $(1,0)$ and another intersection point is of the form $(\cos\theta,\sin\theta)$ for some $0<\theta\leq \pi/2.$  Observe that the arc of $\partial K$ between these touch points is contained in the square $[0,1] \times [0,1].$ \\
Parameterize $\partial K$ by the spanned area, i.e. by a curve
$p\mapsto \gamma(p)$ such that $[\gamma, \dot{\gamma}]=1$. Therefore we have
$[\gamma, \ddot{\gamma}]=0$, hence
$\ddot{\gamma}(p)=-\kappa_0(p)\gamma(p)$, for all $p$, where
$\kappa_0(p)$ is precisely the centro-affine curvature along the boundary of $K$. Let
$\gamma(p)=(x(p),y(p))$, then $\ddot{x}(p)=-\kappa_0(p)x(p)$ and
$\ddot{y}(p)=-\kappa_0(p)y(p).$ Suppose that $M=\sup
\kappa_0(p)<1.$ Since $x(0)=1$, $\dot{x}(0)=0$, $y(0)=0$ and
$\dot{y}(0)=1,$ a standard comparison theorem for equations of the
form $\ddot{x}=-a^2x$ implies that $x(p)> \cos p$ and $y(p)>\sin p$
for all $p\in(0,\pi/2]$. Therefore, $x(p)^2+y(p)^2>1$ for all
$p\in(0,\pi/2]$. This means that $\gamma$ leaves the square
$[0,1]\times [0,1]$ before it has a chance to touch the circle again, contradicting our assumption.

\end{proof}
\begin{theorem}\label{thm: semi main} Suppose that $\tilde{s}_t$ is a solution of the normalized flow (\ref{e: normalized flow}) for some initial convex body in ${\mathcal{K}}_{sym}$ and that  $\{t_k\}_k$  is the sequence of times realizing the limit (\ref{eq:sup}) in Corollary \ref{cor: limsup idea}.  Then there exist two families of centered ellipses $\{\mathcal{E}_{in}(t_k)\}$, $\{\mathcal{E}_{out}(t_k)\}$ such that \begin{equation}\label{ie: inclusion}
{\mathcal{E}_{in}(t_k)}\subseteq \tilde{K}_{t_k}\subseteq {\mathcal{E}_{out}(t_k)}
\end{equation}
with
\begin{equation}\label{e: limit of difference}
\limsup_{t_k\to T}||s_{\mathcal{E}_{out}(t_k)}-s_{{\mathcal{E}_{in}(t_k)}}||=0.
\end{equation}
Furthermore, the sequence of curves $\partial \tilde{K}_{t_k}$  converge,  in the Hausdorff metric, to the unit circle modulo $SL(2)$.
\end{theorem}
\begin{proof}
By Corollary \ref{cor: limit of affine support}, we have
\begin{equation}\label{e:normalized ratio}
\lim_{t_k\to T}
\left(\frac{\tilde{\kappa}}{\tilde{s}^3}\right)(\theta,{t_k})=1.
\end{equation}

Thus, the first half of the claim follows from Lemma  \ref {lem: ellipsoid app}.\\
Now we proceed to prove the second half of the claim.

  Evidently we can find an appropriate family of special linear transformations $\{L_{t_k}\}_{t_k}$ such that $L_{t_k}(\mathcal{E}_{out}({t_k}))$ is a circle at each time ${t_k}$. Each such area preserving linear transformation $L_{t_k}$ minimizes the Euclidean length of the ellipse $\mathcal{E}_{out}({t_k})$ at time $t_k$.

  Thus, the construction of $\mathcal{E}_{out}({t_k}), \mathcal{E}_{in}({t_k})$ implies
  $$\lim_{t_k\to T}L_{t_k}(\mathcal{E}_{out}({t_k}))=\lim_{t_k\to T}L_{t_k}(\mathcal{E}_{in}({t_k}))=\mathbb{S}^1$$
   in the Hausdorff metric:\\ Recall, from
   Lemma \ref{lem: ellipsoid app}, that
   $$\min_{\theta\in \mathbb{S}^1}
\left(\frac{\tilde{\kappa}}{\tilde{s}^3}\right)(\theta,{{t_k}})=\frac{\kappa}{s^3}(\mathcal{E}_{out}({t_k})).$$
Since $\displaystyle \frac{\kappa}{s^3}$ is invariant under $SL(2)$, we have $\displaystyle \frac{\kappa}{s^3}(\mathcal{E}_{out}({t_k}))=\frac{\kappa}{s^3}(L_{t_k}(\mathcal{E}_{out}({t_k})))$, therefore $\displaystyle \lim_{{t_k}\to T}\frac{\kappa}{s^3}(L_{t_k}(\mathcal{E}_{out}({t_k})))=1$. This implies $\lim_{t_k\to T}L_{t_k}(\mathcal{E}_{out}({t_k}))=\mathbb{S}^1$ in the Hausdorff metric.
Similarly, from the choice of $\mathcal{E}_{in}({t_k})$ in Lemma \ref{lem: ellipsoid app}, we have that
  $$\max_{\theta\in \mathbb{S}^1}
\left(\frac{\tilde{\kappa}}{\tilde{s}^3}\right)(\theta,{{t_k}})=\frac{\kappa}{s^3}(\mathcal{E}_{in}({t_k})),$$
 therefore $\displaystyle \lim_{{t_k}\to T}\frac{\kappa}{s^3}(L_{t_k}(\mathcal{E}_{in}({t_k})))=1.$
 This implies $\lim_{{t_k}\to T}A(L_{t_k}(\mathcal{E}_{in}({t_k})))=\pi $. As $L_{t_k}(\mathcal{E}_{in}({t_k}))\subseteq L_{t_k}(\mathcal{E}_{out}({t_k}))$, we conclude that
$\lim_{{t_k}\to T}L_{t_k}(\mathcal{E}_{in}({t_k}))=\mathbb{S}^1$ in the Hausdorff metric.\\
  Now, applying $\{L_{t_k}\}_{t_k}$ to the inclusions (\ref{ie: inclusion}), we obtain that $L_{t_k}(\tilde{K}_{t_k})$ converges to the unit disk in the Hausdorff metric.
\end{proof}

\begin{corollary}\label{cor: lim of volume product}
Along the flow (\ref{e: normalized flow}) with an arbitrary initial condition in $\mathcal{K}_{sym}$,  we have
\begin{align*}
\lim_{t\to T}A(K_t)A(K_t^{\circ})=\pi^2
\end{align*}
for $p>1.$
\end{corollary}
\begin{proof}
Recall that the area product $A(K_t)A(K_t^{\circ})$ is invariant under the general linear group, $GL(2)$, and increasing along any $p$-flow, unless the  boundaries of the evolving convex bodies are centered ellipses. Moreover, as the convex bodies are centrally symmetric with the center of symmetry at the origin, the  Santal\'o inequality gives  $A(t)A^{\circ}(t)\leq\pi^2$  with equality if and only if the boundary curves are ellipses centered at the origin,  \cite{Santalo}. Consequently, Theorem \ref{thm: semi main} implies the claim.
\end{proof}
Before stating our main theorems let us recall the following frequently used fact in both convex geometry and analysis of PDEs which is due to Fritz John 1948.
\begin{theorem}\label{John}\cite{John}. John's Inclusion. Suppose $K$ is a convex body in $\mathbb{R}^n$ then there is a unique ellipse $\mathcal{E}_J$ of maximal volume contained in $K$. Furthermore, if $K$ is origin symmetric then
$$\mathcal{E}_J\subseteq K\subseteq\sqrt{n}\mathcal{E}_J.$$
\end{theorem}
 Theorem \ref{John} immediately implies if $K$ is an origin symmetric convex body such that whose volumes is $\omega_n$, volume of the unit ball in $\mathbb{R}^n$, then there is an affine transformation $L$ such that $r^{+}({LK})\leq \sqrt{n}$ and $r_{-}({{LK}})\ge \frac{1}{\sqrt{n}}$, where $r^{+}(LK)$ and $r_{-}({LK})$ is the inner radius and outer radius of $LK$ respectively.
Now, we are ready to prove one of the main theorems:
\begin{theorem}\label{thm: main}
Let $p>1$. Suppose $\tilde{K}_t$ is a solution of the normalized flow (\ref{e: normalized flow}) for some initial convex body in ${\mathcal{K}}_{sym}$. Then there exists a family of special linear transformations  $\{L_t\}_{t\in[0,T)}\subset SL(2)$ such that $L_t(\partial\tilde{K}_t)$ converges to $\mathbb{S}^1$ in the Hausdorff metric. \label{theorem: previous}
\end{theorem}
\begin{proof}
At each time $t$, we apply a special linear transformation $L_t$ such that the Euclidean length of $\partial{\tilde{K}}_t$ is minimized. Let $\{t_i\}_i$ be a sequence of times converging to $T$. John's Inclusion or Proposition 8 of \cite{BA3} implies the compactness of the set of convex bodies $L_{t_i}({\tilde{K}}_{t_i})$. By Corollary \ref{cor: lim of volume product} and Blaschke Selection Theorem,  each subsequence of $L_{t_i}(\partial{\tilde{K}}_{t_i})$ has a subsequence $L_{t_{i_j}}(\partial{\tilde{K}}_{t_{i_j}})$ such that $L_{t_{i_j}}(\partial{\tilde{K}}_{t_{i_j}})$ converges,  in the Hausdorff metric, to an ellipse of enclosed area $\pi$. Thus, the length minimization condition rules out the degeneracy of the limit ellipse and, in fact, it  implies that $L_{t_{i_j}}(\partial{\tilde{K}}_{t_{i_j}})$ converges to the unit circle in the Hausdorff topology.
\end{proof}

\section{Expanding $p$-flow }
\begin{lemma}\label{lem: dual flow}
As $K_t$ evolves by the centro-affine curvature flow (\ref{e: asli}), its dual $K^{\circ}_t$ evolves, up to a diffeomorphism
, under the flow
\begin{equation}\label{e: dual asli 2}
\frac{\partial}{\partial
t}s=s\left(\frac{\kappa}{s^3}\right)^{-\frac{p}{p+2}},\
s(.,t)=s_{\partial K^{\circ}_{t}}(.),\ s(.,0)=s_{\partial K^{\circ}_{0}}(.).
\end{equation}
\end{lemma}
\begin{proof}
The proof of Lemma \ref{lem: dual flow} is given in \cite{S}, but, for
completeness, we'll  present it here. Recall that $\displaystyle A(K^{\circ})=\frac{1}{2}\int_{\mathbb{S}^1}\frac{1}{s^2}d\theta$
and that, under sufficient regularity assumptions on $\partial K$ which are satisfied here, $\Omega_{q}(K)=\Omega_{\frac{4}{q}}(K^{\circ})$ for any $q \neq -n$, in which case the $q$-affine length is  not defined. Therefore, as
$K_t$ evolves by the centro-affine curvature flow (\ref{e: asli}),
the volume of the dual body $K^{\circ}_t$ changes by
$$\frac{d}{dt}A^{\circ}(t)=\Omega_{-\frac{p}{p+1}}^{\circ}(t),$$
 where the notation stands for
$\Omega_{-\frac{p}{p+1}}(K^{\circ}_t).$
Compared with the rate of change of the area of a convex body $L$
whose boundary is deformed by a normal vector field with speed $v$,
which is
$\displaystyle \frac{d}{dt}A(L)=\int_{\mathbb{S}^1}v\frac{1}{\kappa_L}d\theta$, we
infer that while $K_t$ evolves, up to a diffeomorphism, by (\ref{e: asli}), its dual
$K^{\circ}_t$ evolves, up to a diffeomorphism, by (\ref{e: dual asli 2}).

\end{proof}
Similar to Propositions \ref{prop: short time existence and uniqueness} and \ref{prop: containment principle} of \cite{S}, we have
\begin{proposition}\label{prop: dual short time existence}
Let $K_0$ be a convex body belonging to $\mathcal{K}_{sym}$ and let
$p\ge 1$. Then, there exists a time $T>0$ for which equation (\ref{e: dual asli 2})
 has a unique  solution starting from $K_0.$
\end{proposition}
\begin{proposition}\label{prop: dual containment principle}
Containment Principle. If $K_{in}$ and $K_{out}$ are the two convex
bodies in $\mathcal{K}_{sym}$ such that $K_{int}\subset K_{out}$,
and $p\ge 1$, then $K_{in}(t)\subseteq K_{out}(t)$ for as long as the
solutions $K_{in}(t)$ and $K_{out}(t)$ (with given initial data
$K_{in}(0)=K_{in}$, $K_{out}(0)=K_{out}$)  of (\ref{e: dual asli 2})
exist in $\mathcal{K}_{sym}.$
\end{proposition}
Similar to Lemma \ref{lem: symmetry preserved}
we have
\begin{lemma}\label{lem: dual symmetry preserved}Let $\{K_t\}_t$ be a solution of (\ref{e: dual asli 2}) where $K_0\in\mathcal{K}_{sym}.$
Then, $K_t\in\mathcal{K}_{sym}$ as long as the flow exists.
\end{lemma}
Combining Proposition \ref{prop: length}, Lemma \ref{lem: dual
flow}, Proportions \ref{prop: dual short time existence} and
\ref{prop: dual  containment principle}  we obtain
\begin{proposition}
Suppose $K_t$ is a family of convex bodies such that it evolves
under the flow
 $$\frac{\partial}{\partial t}s(.,t)=s\left(\frac{\kappa}{s^3}\right)^{-\frac{p}{p+2}}(.,t)$$
with $p\ge 1$.
Then
$$\forall \, \theta: \ \ \ \lim_{t\to T}s(\theta, t)=\infty.$$
\end{proposition}

\begin{proposition}\label{prop: dual homothetic}
Ellipses centered at the origin are the only homothetic solutions to (\ref{e: dual asli 2}).
\end{proposition}
\begin{proof}
The proof follows from the duality between the two flows and Proposition \ref{prop: volume product}.
\end{proof}

Furthermore, we obtain:
\begin{theorem}\label{thm: main dual}
Let $p>1$. Suppose $\tilde{K}_t$ is a solution of the normalized flow derived from (\ref{e: dual asli 2}) for some initial convex body in ${\mathcal{K}}_{sym}$. Then there exists a family of linear transformations  $\{L_t\}_{t\in[0,T)}\subset SL(2)$ such that $L_t(\partial\tilde{K}_t)$ converge to $\mathbb{S}^1$
in the Hausdorff metric.
\end{theorem}
\begin{proof}
Let $\{L_t\}_{t\in[0,T)}$ be the family of length minimizing special linear transformations that we defined in the proof of Theorem \ref{theorem: previous}. Since $(L_t(\tilde{K}_t))^{\circ}=L_t^{-t}(\tilde{K}_t^{\circ})$, where $L_t^{-t}$ is the inverse transpose of $L_t,$ the claim follows.
\end{proof}
\section{A proof of $p$-affine isoperimetric inequality}
In this section we aim to provide a new proof of $p$-affine isoperimetric inequalities, $p\geq1$, for a convex body $K\in \mathcal{K}_{sym}$. Since our proofs of Theorems \ref{thm: main}, and \ref{thm: main dual} are dependant on $p$-affine isoperimetric inequalities we can't apply our results on $p$-affine flows to obtain $p$-affine isoperimetric inequalities. Instead we employ the affine normal flow to reach our goal, see \cite{BA4}.
\\  We state the following general evolution equation for $\Omega_{l}$ under the contracting $p$-affine flow  for any $l\in \mathbb{R}:$
\begin{equation}\label{e: eveq general}
\frac{d}{dt}\Omega_l(t)=\frac{2(l-2)}{l+2}\int_{\gamma}\sigma^{1-\frac{3p}{p+2}-\frac{3l}{l+2}}d\mathfrak{s}
+\frac{18pl}{(l+2)^2(p+2)}\int_{\gamma}\sigma^{-\frac{3p}{p+2}-\frac{3l}{l+2}}\sigma_{\mathfrak{s}}^2d\mathfrak{s}.
\end{equation}
Proof of this equation is parallel to the one of part four of Lemma \ref{lemma:4parts}.

\begin{lemma}\label{lem: controlling derivative of $l$-affine surface area along affine normal flow}
 The following sharp affine isoperimetric inequalities hold along the affine normal flow. \\
If $1\leq l\leq 2$, then
\begin{align*}
\frac{d}{dt}\Omega_l(t)\geq \frac{l-2}{l+2}\frac{\Omega_l\Omega_1}{A}+\frac{2(l-1)(4l^2+3l+2)}{(l+2)^3}\int_{\gamma}\sigma^{-1-\frac{3l}{l+2}}\sigma_{\mathfrak{s}}^2d\mathfrak{s},
\end{align*}
while, if $l\geq2$, we then have
\begin{align*}
\frac{d}{dt}\Omega_l(t)\geq\frac{l-2}{l+2}\frac{\Omega_l\Omega_1}{A}
+\frac{6l}{(l+2)^2}\int_{\gamma}\sigma^{-1-\frac{3l}{l+2}}\sigma_{\mathfrak{s}}^2d\mathfrak{s},
\end{align*}
\end{lemma}
\begin{proof}
Before presenting a proof of the second claim let us to state the following generalized H\"{o}lder inequality developed by Andrews \cite{BA1}. If $M$ is a compact manifold with a volume form $d\omega,$ $g$ is a continues function on $M$ and $F$ is decreasing real, positive function, then
$$\frac{\int_{M}gF(g)d\omega}{\int_{M}F(g)d\omega}\leq \frac{\int_{M}gd\omega}{\int_{M}d\omega}.$$
If $F$ is strictly decreasing, then equality occurs if and only if $g$ is constant.\\
Define $d\omega=\sigma d\mathfrak{s}$, $g=\sigma$ and $F(x):=x^{-\frac{3l}{l+2}}$. Furthermore, observe that for a convex body $K$ in $\mathbb{R}^2$ we have $2A=\int_{\partial K}\sigma d\mathfrak{s}.$ This implies
$$\int_{\partial K}\sigma^{-\frac{3l}{l+2}}d\mathfrak{s}\geq \frac{\Omega_l\Omega_1}{2A},$$
hence the second claim follows by this last inequality and the evolution equation (\ref{e: eveq general}) for $p=1.$
To prove the first inequality, one can proceed similarly as in the proof of inequality (\ref{e: strong affine isoperimetric inequality1}),
and use the affine-geometric Wirtinger inequality developed by Andrews, Lemma 6, \cite{BA3}.

\end{proof}
\begin{lemma}\label{lem: l-ratio}
Let $l\geq1$ then, the $l$-affine isoperimetric ratio,
$\frac{\Omega_l^{2+l}(t)}{A^{2-l}(t)}$, is non-decreasing along the
 affine normal flow and remains constant if and only if $K_t$ is an origin centered ellipse.
\end{lemma}
\begin{proof}
\begin{align*}
\frac{d}{dt}\frac{\Omega_l^{2+l}(t)}{A^{2-l}(t)}&=\frac{(2+l)\Omega_l^{l+1}(t)A^{2-l}(t)\frac{d}{dt}\Omega_l+
(2-l)A^{1-l}(t)\Omega_l^{2+l}(t)\Omega_p(t)}{A^{2(2-l)}(t)}
\\ &=\frac{\Omega_l^{l+1}(t)}{A^{2-l}(t)}\left((2+l)\frac{d}{dt}\Omega_l-(l-2)\frac{\Omega_l(t)\Omega_1(t)}{A(t)}\right)\geq
0,
\end{align*}
where we used Lemma \ref{lem: controlling derivative of $l$-affine surface area along affine normal flow} on the last line.
\end{proof}
\begin{theorem}Let $l\geq1$, then the following $l$-affine isoperimetric inequality holds for a convex body $K\in\mathcal{K}_{sym}$.
$$\frac{\Omega_l^{2+l}(K)}{A^{2-l}(K)}\leq 2^{l+2}\pi^{2l},$$
moreover, equality holds if and only if for centered ellipses at origin.
\end{theorem}
\begin{proof}
The claim is an immediate consequence of the weak convergence of the solutions of the normalized affine normal flow to a centered ellipse and Lemma \ref{lem: l-ratio}.
\end{proof}

\textbf{Acknowledgment.}\\
 I would like to immensely thank Alina Stancu for many useful discussions on this work and for her help and advice on various parts of this paper.
 I would like to thank the referee for helpful comments and suggestions.

\bibliographystyle{amsplain}

\end{document}